\documentclass[reqno]{amsart}
\usepackage{epsfig,amssymb,amsmath,mathrsfs}
\usepackage{hyperref}
\newcommand{\sM}{\mathscr{M}}
\newcommand{\R}{\mathbb{R}}
\newcommand{\N}{\mathbb{N}}
\newcommand{\Z}{\mathbb{Z}}
\newcommand{\union}{\cup}

\newcommand\dist{d}
\newcommand\dihed{A_k}
\newcommand\prodd{\operatorname{prod}}
\newcommand\poss{\operatorname{Pos}}
\newcommand\negg{\operatorname{Neg}}
\newcommand\dualposs{\operatorname{\widetilde Pos}}
\newcommand\dualnegg{\operatorname{\widetilde Neg}}
\newcommand\dualZ{\tilde Z}

\newcommand\closure{\operatorname{cl}}

\newcommand\dists{\mathcal{D}}
\newcommand\dualdists{\mathcal{\tilde D}}
\newcommand\dualdist{\tilde d}
\newcommand\gens{S}
\newcommand\dualgens{\tilde S}
\newcommand\dualcompactification{\tilde\sM}
\newcommand\dualfinites{\tilde Z_0}

\newcommand\plusclass{+\hat\infty}
\newcommand\minusclass{-\hat\infty}
\newcommand\grote{\mathcal{G}}
\newcommand\conda{i}
\newcommand\condb{ii}
\newcommand\condc{iii}
\newcommand\condd{iv}

\newtheorem*{theorem*}{Theorem}
\newtheorem*{proposition*}{Proposition}

\newtheorem{prop}{Proposition}[section]
\newtheorem{proposition}[prop]{Proposition}

\newtheorem{lemma}[prop]{Lemma}

\newtheorem{theorem}[prop]{Theorem}
\theoremstyle{definition}

\begin{document}

\title{Busemann points of Artin groups of dihedral type}
\date{\today}
\author{Cormac Walsh}
\address{INRIA, Domaine de Voluceau,
78153 Le Chesnay C\'edex, France}
\email{cormac.walsh@inria.fr}



\begin{abstract}
We study the horofunction boundary of an Artin group of dihedral type
with its word metric coming from either the usual Artin generators
or the dual generators.
In both cases, we determine the horoboundary and say
which points are Busemann points, that is the limits of geodesic rays.
In the case of the dual generators, it turns out
that all boundary points are Busemann points, but this
is not true for the Artin generators.
We also characterise the geodesics
with respect to the dual generators, which allows us
to calculate the associated geodesic growth series.
\end{abstract}

\maketitle

\section{Introduction}

Consider the following metric space boundary,
defined first by Gromov~\cite{gromov:hyperbolicmanifolds}.
One assigns to each point $z$ in the metric space $(X,d)$ the function
$\psi_z:X\to \R$,
\begin{equation*}
\psi_z(x) := d(x,z)-d(b,z),
\end{equation*}
where $b$ is some basepoint.
If $X$ is proper and complete, then the map
$\psi:X\to C(X),\, z\mapsto \psi_z$ defines an embedding of $X$ into $C(X)$,
the space of continuous real-valued functions on $X$ endowed
with the topology of uniform convergence on compacts.
The horofunction boundary is defined to be
$X(\infty):=\closure\{\psi_z\mid z\in X\}\backslash\{\psi_z\mid z\in X\}$,
and its elements are called horofunctions.

This boundary is not the same as the better known Gromov boundary of
a $\delta$-hyperbolic space. For these spaces, it has been
shown~\cite{coornaert_papadopoulos_horofunctions,winweb_hyperbolic,
storm_barycenter} that the horoboundary is finer than the Gromov boundary
in the sense that there exists an equivariant continuous surjection from
the former to the latter.

An interesting class of metric spaces are the Cayley graphs of finitely
generated groups with their word metric.
Here one may hope to have a combinatorial description of the horoboundary.
Rieffel~\cite{rieffel_group} has investigated the horoboundary in this
setting. A length function on a discrete group naturally
gives rise to a metric on the state space of the reduced group
C*-algebra~\cite{connes}
and, in the case of $\Z^d$ with a word metric coming from a finite set
of generators, Rieffel used the horoboundary to determine
certain properties of this metric, in particular, whether it is
compatible with the weak* topology on the state space.

This motivates the study of the horoboundary of other finitely generated groups.
In this paper, we investigate the horofunction boundary of
the Artin groups of dihedral type.
Let $\prodd(s,t;n):= ststs\cdots$, with $n$ factors in the product.
The Artin groups of dihedral type have the following presentation:
\begin{align*}
\dihed = \langle a,b \mid \prodd(a,b;k)=\prodd(b,a;k) \rangle,
\qquad\text{with $k\ge3$.}
\end{align*}
Observe that $A_3$ is the braid group on three strands.
The generators traditionally considered are the Artin generators
$\gens:=\{a,b,a^{-1},b^{-1}\}$.

In what follows, we will have need of the Garside normal form
for elements of $\dihed$.
The element $\Delta:=\prodd(a,b;k)=\prodd(b,a;k)$ is called the
Garside element.
Let
\begin{align*}
M^+:= \{a,b,ab,ba,\dots,\prodd(a,b;k-1),\prodd(b,a;k-1)\}.
\end{align*}
It can be shown~\cite{word_processing} that $w\in \dihed$ can be written
\begin{align*}
w = w_1\cdots w_n \Delta^r
\end{align*}
for some $r\in\Z$ and $w_1,\dots,w_n\in M^+$.
This decomposition is unique if $n$ is required to be minimal.
We call it the right normal form of $w$.
The factors $w_1,\dots,w_n$ are called the canonical factors of $w$.

One can also write $w$ in left normal form:
$w=\Delta^r w'_1\cdots w'_n$, with $r\in\Z$ and $w'_1,\dots,w'_n\in M^+$.

To calculate the horoboundary, we will need a formula for the word length
metric. An algorithm was given in~\cite{berger_braids} for
finding a geodesic word representing any given element of $A_3$.
In~\cite{sabalka_geodesics}, there is a criterion for when a word
is a geodesic in $A_3$.
Both these results were generalised in~\cite{mairesse_matheus_growth} to
arbitrary $k\ge3$.
It was shown that a freely reduced word $u$
is a geodesic with respect to the Artin generators if and only if
\begin{align}
\label{geo}
\poss(u)+\negg(u)\le k.
\end{align}
Here $\poss(u)$ is the length of the longest possible element of
$M^+\union\{\Delta\}$ obtainable by multiplying together consecutive letters
of $u$.
The length of an element $\prodd(a,b;n)$ or $\prodd(b,a;n)$
of $M^+\union\{\Delta\}$ is defined to be $n$.
Likewise, $\negg(u)$ is the length of the longest possible element of
$M^-\union\{\Delta^{-1}\}$ obtainable in the same way,
where $M^-:= (M^+)^{-1}$.

We use the algorithm in~\cite{mairesse_matheus_growth} to find a simple
formula for the word length metric.
\begin{proposition*}
Let $x=\Delta^{r} z_1\cdots z_m$ be an element of $\dihed$ written in left
normal form. Let $(p_0,\dots,p_{k-1})\in\N^{k}$ be such that
$p_0:=r$ and, for each $i\in\{1,\dots,k-1\}$, $p_i-p_{i-1}= m_{k-i}$,
where $m_i$ is the number of canonical factors of $x$ of length $i$.
Then the distance from the identity $e$ to $x$ in the Artin-generator
word-length metric is
\begin{align*}
\dist(e,x)=\sum_{i=0}^{k-1} |p_i|.
\end{align*}
\end{proposition*}
Since $d$ is invariant under left multiplication, that is,
$d(y,x)=d(e,y^{-1}x)$, we can use this formula to calculate the distance
between any pair of elements $y$ and $x$ of $\dihed$.
With this knowledge we can find the following description of the horofunction
compactification.
\newcommand\emme{m}

Let $Z$ be the set of possibly infinite words of positive generators having
no product of consecutive letters equal to $\Delta$.
We can write each element $z$ of $Z$ as a concatenation of substrings in
such a way that the products of the letters in every substring equals
an element of $M^+$ and the combined product of letters in each consecutive
pair of substrings is not in $M^+$. Because $z$ does not contain $\Delta$,
this decomposition is unique. Let $\emme_i(z)$ denote the number of substrings
of length $i$. Note that if $z$ is an infinite word, then this number will
be infinite for some $i$.

Let $\Omega'$ denote the set of $(p,z)$ in
$(\Z\union\{-\infty,+\infty\})^k\times Z$
satisfying the following:
\begin{itemize}
\item
$p_i-p_{i-1}\ge \emme_{k-i}(z)$ for all $i\in\{1,\dots,k-1\}$ such that
$p_i$ and $p_{i-1}$ are not both $-\infty$ nor both $+\infty$;
\item
if $z$ is finite, then $p_i-p_{i-1} = \emme_{k-i}(z)$
for all $i\in\{1,\dots,k-1\}$ such that
$p_i$ and $p_{i-1}$ are not both $-\infty$ nor both $+\infty$.
\end{itemize}
We take the product topology on $\Omega'$.

We now define $\Omega$ to be the quotient topological space of $\Omega'$
where the elements of $(+\infty,\dots,+\infty)\times Z$ are considered
equivalent and so also are those in $(-\infty,\dots,-\infty)\times Z$.
We denote these two equivalence classes by $\plusclass$ and $\minusclass$,
respectively.

We let $\mathcal{M}$ denote the horofunction compactification of $\dihed$
with the Artin-generator word metric.
The basepoint is taken to be the identity.

\begin{theorem*}
The sets $\Omega$ and $\mathcal{M}$ are homeomorphic.
\end{theorem*}
Let $Z_0$ be the set of elements of $Z$ that are finite words.
Let $\Omega_0$ denote the set of $(p,z)$ in $\Z^k\times Z_0$
such that $p_i-p_{i-1}= \emme_{k-i}(z)$ for all $i\in\{1,\dots,k-1\}$.
We will show that the elements of $\Omega_0$ are exactly the elements of
$\Omega$ corresponding to functions of the form $d(\cdot,z)-d(e,z)$
in $\mathcal{M}$.

Of particular interest are those horofunctions that are the limits of
almost-geodesics; see~\cite{AGW-m} and~\cite{rieffel_group} for two related
definitions of this concept.
Rieffel calls the limits of such paths Busemann points.
In the present context, since the metric takes only integer values,
the Busemann points are exactly the limits of geodesics
(see~\cite{winweb_busemann}).
Develin~\cite{develin_cayley}, investigated the horoboundary of finitely
generated abelian groups with their word metrics and showed that all their
horofunctions are Busemann.
Webster and Winchester~\cite{winweb_busemann} gave a necessary and sufficient
condition for all horofunctions of a finitely generated group to be Busemann.

We prove the following characterisation of the Busemann points of $\dihed$.
\begin{theorem*}
A function in $\mathcal{M}$ is a Busemann point if and only if the
corresponding element $(p,z)$ of $\Omega$ is in $\Omega\backslash\Omega_0$
and satisfies the following:
$p_i-p_{i-1} = \emme_{k-i}(z)$ for every $i\in\{1,\dots,k-1\}$ such that
$p_i$ and $p_{i-1}$ are not both $-\infty$ nor both $+\infty$.
\end{theorem*}

The group $A_k$ also has a dual presentation:
\begin{align*}
A_k = \langle \sigma_1,\dots,\sigma_k
 \mid \sigma_1\sigma_2 = \sigma_2\sigma_3 = \cdots = \sigma_k\sigma_1 \rangle,
\qquad\text{with $k\ge3$.}
\end{align*}
The set of dual generators is
$\dualgens := \{\sigma_1,\dots,\sigma_k,\sigma_1^{-1},\dots,\sigma_k^{-1}\}$.
These are related to the Artin generators in the following way: $\sigma_1=a$,
$\sigma_2=b$, and
\begin{align*}
\sigma_j = \begin{cases}
\prodd(b^{-1},a^{-1};j-2) \prodd(a,b;j-1), & \text{if $j$ is odd,} \\
\prodd(b^{-1},a^{-1};j-2) \prodd(b,a;j-1), & \text{if $j$ is even,}
\end{cases}
\end{align*}
for $j\in\{3,\dots,k\}$.
The existence of a dual presentation holds more generally for all Artin groups
of finite type~\cite{bessis_dual}.

There are also Garside normal forms related to the dual presentation.
Here the Garside element is
$\delta:= \sigma_1 \sigma_2 = \cdots = \sigma_{k} \sigma_1$.

Again, we find a formula for the word length metric.
\begin{proposition*}
Let $w=\delta^{r} w_1 w_2\cdots w_s$ be written in left normal
form. Then the distance between the identity and $w$ with respect to the
dual generators is given by $\dualdist(e,w)=|r|+|r+s|$.
\end{proposition*}

Using this formula, we again determine the horoboundary. This time however,
there are no non-Busemann points.
\begin{theorem*}
In the horoboundary of $\dihed$ with the dual-generator word metric,
all horofunctions are Busemann points.
\end{theorem*}

In general, one would expect the properties of the horofunction boundary
of a group with its word length metric to depend strongly on the generating
set. It would be interesting to know for which groups and for which properties
there is not this dependence.
As already mentioned, all boundary points of abelian groups are
Busemann no matter what the generating set~\cite{develin_cayley}.
On the other hand, the above results show that
for Artin groups of dihedral type
the existence of non-Busemann points depends on the generating set.

We use our formula to establish a criterion for a word to be a geodesic
with respect to the dual generators.
For every word $y$ with letters in $\dualgens$,
let $\dualposs(y)$ be the longest
element of $\{\sigma_1,\dots,\sigma_{k},\delta\}$ obtainable by
multiplying together consecutive letters of $y$.
The generators $\sigma_1,\dots,\sigma_{k}$ are considered to each have
length 1 whereas $\delta$ is considered to have length 2.
Similarly, $\dualnegg(y)$ is defined to be the longest element of
$\{\sigma^{-1}_1,\dots,\sigma^{-1}_{k},\delta^{-1}\}$ obtainable in the
same way.
\begin{proposition*}
Let $y$ be a freely reduced word of dual generators.
Then $y$ is a geodesic if and only if
$\dualposs(y)+\dualnegg(y) \le 2$.
\end{proposition*}

The geodesic growth series of a finitely generated group $G$ with
respect to a generating set $S$ is
\begin{align*}
\grote_{(G,S)}(x) := \sum_{n=0}^\infty a_n x^n,
\end{align*}
where $a_n$ is the number of words of length $n$ that are geodesic
with respect to $S$.

It is obvious from the characterisation of geodesics given above that the
set of geodesic words with respect to the dual set of generators $\dualgens$
is a regular language. It follows that the geodesic growth series is
rational~\cite{word_processing}, that is, can be expressed as the
quotient of two integer-coefficient polynomials in the ring of
formal power series $\Z[[x]]$.
We calculate this growth series explicitly.
\begin{theorem*}
The geodesic growth series of $\dihed$ with the dual generators is
\begin{align*}
\grote(x) =
      \frac{1+(3-2k)x + (2 + k^2 -3k) x^2 - 2k(k-1)x^3}
           {(1-kx)(1-2(k-1)x)(1-(k-1)x)}.
\end{align*}
\end{theorem*}

The geodesic growth series has previously been determined for $A_k$ with
other generating sets. Charney and Meier~\cite{charney_meier_language}
calculate it for the generating sets
$\{\sigma_1^\pm,\dots,\sigma_k^\pm,\delta^\pm\}$
and $M^+ \union M^- \union \{\Delta^\pm\}$.
Sabalka~\cite{sabalka_geodesics} calculates it for the 3--strand
braid group $A_3$ with the Artin generators,
a result which was generalised by Mairesse
and Math\'eus~\cite{mairesse_matheus_growth} to $A_k$; $k\ge 3$,
again with the Artin generators.

\section{Artin generators}
\label{sec:artin}

\begin{proposition}
\label{prop:artindist}
Let $x=\Delta^{r} z_1\cdots z_m$ be an element of $\dihed$ written in left
normal form. Let $(p_0,\dots,p_{k-1})\in\Z^{k}$ be such that
$p_0:=r$ and, for each $i\in\{1,\dots,k-1\}$, $p_i-p_{i-1}= m_{k-i}$,
where $m_i$ is the number of canonical factors of $x$ of length $i$.
Then the distance from the identity $e$ to $x$ in the Artin-generator
word-length metric is
\begin{align*}
\dist(e,x)=\sum_{i=0}^{k-1} |p_i|.
\end{align*}
\end{proposition}
\begin{proof}
In~\cite{mairesse_matheus_growth}, there is an algorithm for finding a
geodesic representative of an element $x$ of $\dihed$ given its
normal form. This algorithm consists, in the case when $r<0$,
of shifting each instance of $\Delta^{-1}$ across and combining it with one
of the canonical factors of longest length. This procedure is continued
until all the $\Delta^{-1}$s have been moved across or there are no
more canonical factors with which to multiply. The resulting word is shown to
be a geodesic representative of $x$.

If $r\ge 0$, then the algorithm leaves the normal form unchanged,
and so
\begin{align*}
\dist(e,x)=\sum_{i=1}^{k-1} im_i+kr = \sum_{i=0}^{k-1} p_i,
\end{align*}
which proves the result in this case since here all the $p_i$ are non-negative.

On the other hand, if $-r\ge \sum_{i=1}^{k-1}  m_i$, then all the canonical
factors are changed: each factor of length $i\in\{1,\dots,k-1\}$ is replaced
by a word of length $k-i$.
Therefore
\begin{align*}
\dist(e,x) &= \sum_{i=1}^{k-1} (k-i)m_i+k\Big(-r-\sum_{i=1}^{k-1} m_i\Big) \\
           &= -kr - \sum_{i=1}^{k-1} i m_i \\
           &= -\sum_{i=0}^{k-1} p_i.
\end{align*}
But in this case all the $p_i$ are non-positive and we conclude that the result
holds here also.

The final case to consider is when $0<-r<\sum_{i=1}^{k-1} m_i$.
In this case, there is some $j\in\N$ such that all factors of length
greater than $j$ are changed, all factors of length less than $j$ are
unchanged, and possibly some factors of length $j$ are changed.
So we have
\begin{align*}
\dist(e,x)
     &= \sum_{i=1}^{j-1} i m_i + j \Big(\sum_{i=j}^{k-1} m_i + r\Big)
   + (k-j) \Big(-r-\sum_{i=j+1}^{k-1} m_i\Big) + \sum_{i=j+1}^{k-1} (k-i)m_i \\
     &= p_{k-1} + \cdots + p_{k-j} - p_{k-j-1} - \cdots - p_0.
\end{align*}
Because of the choice of $j$,
we have $\sum_{i=j}^{k-1} m_i \ge -r \ge \sum_{i=j+1}^{k-1} m_i$,
and so $p_i$ is non-negative for $i\ge k-j$ and non-positive for $i<k-j$.
Therefore the result holds in this case also.
\end{proof}

Motivated by this we define the following map.
Let $z$ be an element of $\dihed$. For each $i\in\{1,\dots,k\}$,
let $m_i$ be the number of canonical factors of length $i$ when $z$ is
written in left normal form. We define $\pi: \dihed\to \Z^k$ by
\begin{align*}
\pi(z) := (m_k, m_k+m_{k-1}, \dots, m_k+\dots+m_1).
\end{align*}
Let $w$ and $z$ be two elements of $\dihed$.
We define
\begin{align*}
\phi(w,z):=\pi(w^{-1}z)-\pi(z).
\end{align*}
For $w\in\dihed$, denote by $\tau(w)$ the conjugate of $w$ by $\Delta$,
that is
\begin{align*}
\tau(w):=\Delta^{-1} w \Delta = \Delta w \Delta^{-1}.
\end{align*}

\begin{lemma}
\label{lem:phiconverges}
Let $w\in\dihed$ and let $z_1 z_2 \cdots$ be an infinite word of positive
generators such that no product of consecutive letters equals $\Delta$.
Then $\phi(w,z_1\cdots z_n)$ converges as $n$ tends to infinity.
\end{lemma}
\begin{proof}
To write $w^{-1}z_1\cdots z_n$ in left normal form,
we first write $w^{-1}$ in left normal form and then repeatedly
take the factors $\Delta$ formed by the joining of $w^{-1}$ and
$z_1\cdots z_n$ out to the left. We obtain something of the form
$\Delta^{r+s}\tau^r(w')z'$, where $r$ is the number of $\Delta$s moved,
$w'$ is a left divisor of $w^{-1}$ and $z'$ is a right divisor of
$z_1\cdots z_n$.
One or both of $w'$ and $z'$ may be the identity.
Since $w$ is of finite length, as $n$ is increased $z'$ must eventually
be different from the identity, and from then on $z'$ will grow in the
same way as
$z_1\cdots z_n$. When $z'$ has grown sufficiently that it contains one of the
canonical factors of $z_1\cdots z_n$, subsequent increases in $n$
will have exactly the same effect on $\pi(w^{-1}z_1\cdots z_n)$ as on 
$\pi(z_1\cdots z_n)$. Therefore $\phi(w,z_1\cdots z_n)$ is eventually constant.
\end{proof}

Recall that $Z$ is the set of possibly infinite words of positive generators
having no product of consecutive letters equal to $\Delta$.
The previous lemma allows us to define $\phi(w,z)$ for $w\in\dihed$ and
$z=z_1 z_2\cdots$ an infinite element of $Z$ to be the limit of
$\phi(w,z_1\cdots z_n)$ as $n$ tends to infinity.

For each $(p,z)\in \Omega'$, define
\begin{align}
\label{psidefinition}
\psi_{p,z} : \dihed \to \Z, \quad
          w\mapsto \sum_{i=0}^{k-1} |p_i+\phi_i(w,z)| - \sum_{i=0}^{k-1} |p_i|.
\end{align}
Note that this formula sometimes requires us to add or subtract infinities.
The convention we shall use will be to separately keep track of the
infinite and finite parts.
Thus $(a\infty + b)+(c\infty+d) = (a+c)\infty+(b+d)$.
Obviously, for $a$ and $b$ finite, $|a\infty+b|$ is equal to $a\infty+b$
if $a> 0$, and is equal to $-a\infty-b$ if $a< 0$.
We see that $\psi_{p,z}$ is always finite because the infinities in the
first term always cancel those in the second.

The following lemma will be needed to show that $\psi$ is constant on the
equivalence classes $\minusclass$ and $\plusclass$.
\begin{lemma}
\label{lem:sumphi}
For all $w$ and $z$ in $\dihed$,
\begin{align*}
\sum_{i=0}^{k-1} \phi_i(w,z) = \sum_{i=0}^{k-1} \pi_i(w^{-1}).
\end{align*}
\end{lemma}
\begin{proof}
Let $y\in\dihed$. Write $y=y_1 \cdots y_s \Delta^r$ in right
normal form and let $m_i$ be the number of canonical factors of length $i$
for each $i\in\{1,\dots,k\}$, so that $m_k=r$.
Consider the effect of left multiplying $y$ by a positive generator $g$.
Either $g$ combines with $y_1$ to form a longer
factor, in which case $m_i$ decreases by one and $m_{i+1}$ increases
by one, where $i$ is the length of $y_1$,
or a new factor is created, in which case $m_1$ increases by one.
In either case, $\sum_{i=0}^{k-1} (\pi_i(gy)-\pi_i(y)) =1$.
We conclude that
\begin{align}
\label{eqn:sumgen}
\sum_{i=0}^{k-1} \phi_i(g^{-1},y) = 1,
\qquad\text{for all $y\in\dihed$ and $g\in\{a,b\}$.}
\end{align}

Similar reasoning shows that
\begin{align}
\label{eqn:sumdelta}
\sum_{i=0}^{k-1} \phi_i(\Delta,y) = -k,
\qquad\text{for all $y\in\dihed$.}
\end{align}

Any $w\in\dihed$ may be written as a product $w_1\cdots w_l$ of negative
generators and copies of $\Delta$. Observe that
\begin{align*}
\phi(w,z) = \phi(w_1,z) + \phi(w_2,w_1^{-1}z)
                  + \dots + \phi(w_l,w_{l-1}^{-1}\cdots w_1^{-1}z).
\end{align*}
Applying~(\ref{eqn:sumgen}) and~(\ref{eqn:sumdelta}), we see that
$\sum_{i=0}^{k-1} \phi_i(w,z)$ is independent of $z$.
Therefore 
\begin{align*}
\sum_{i=0}^{k-1} \phi_i(w,z)
    = \sum_{i=0}^{k-1} \phi_i(w,e)
    = \sum_{i=0}^{k-1} \pi_i(w^{-1}).
\end{align*}
\end{proof}
So we see that if $p$ is identically $-\infty$, then
\begin{align*}
\psi_{p,z}(w) = -\sum_{i=0}^{k-1} \pi_i(w^{-1})
\end{align*}
is independent of $z$. Likewise, if $p$ is identically $+\infty$, then
\begin{align*}
\psi_{p,z}(w) =  \sum_{i=0}^{k-1} \pi_i(w^{-1}).
\end{align*}
We may therefore consider the map $\psi$ to be defined on $\Omega$.

Define $\dists:=\{ \dist(\cdot,x)-\dist(e,x) \mid x\in \dihed \}$.

\begin{lemma}
\label{lem:bijection}
Restricted to $\Omega_0$, the map $\psi$ is a bijection between $\Omega_0$
and $\dists$.
\end{lemma}
\begin{proof}
Let $(p,z)\in\Omega_0$.
Observe that $p_i=\pi_i(z)+p_0 = \pi_i(z\Delta^{p_0})$ for all $0\le i\le k-1$.
For each $w\in \dihed$,
\begin{align*}
\psi_{p,z}(w)
      &= \sum_{i=0}^{k-1}
           |p_i + \pi_i(w^{-1} z) - \pi_i(z)| - \sum_{i=0}^{k-1} |p_i| \\
      &= \sum_{i=0}^{k-1}
           |\pi_i(w^{-1} z \Delta^{p_0})|
           - \sum_{i=0}^{k-1} |\pi_i(z \Delta^{p_0})| \\
      &= \dist(w,z\Delta^{p_0}) - \dist(e,z\Delta^{p_0}).
\end{align*}
The result now follows from the fact that every element of $\dihed$ can be
written in a unique way as $z\Delta^{p_0}$ with $z\in Z_0$ and $p_0\in \N$
and that the $p_i;1\le i\le k-1$ are determined by $z$ and $p_0$ for each
$(p,z)$ in $\Omega_0$. 
\end{proof}

\begin{lemma}
\label{lem:closure}
The set $\Omega_0$ is dense in $\Omega$.
\end{lemma}
\begin{proof}
Clearly, $\minusclass$ and $\plusclass$ are in the closure of $\Omega_0$ since
they are the limits, respectively, of $(-n,\dots,-n,e)$ and $(n,\dots,n,e)$,
where $e$ denotes the empty word.

Let $(p,z)\in \Omega\backslash\{\minusclass,\plusclass\}$ and fix $n\in\N$.
Let $x_n$ be the product of the first
$n$ canonical factors of $z$. Define $b_k:=\max(\min(p_0,n),-n)$ and
$b_{k-i}:=\min(p_i-p_{i-1},n)$ for each $i\in\{1,\dots,k-1\}$.
Let $m_i$ denote the number of canonical factors of length $i$ in $x_n$.

For each $i\in\{1,\dots,k-1\}$, we have that $m_i$ is no greater than
the number of canonical factors of length $i$ in $z$, which is no greater than
$p_{k-i}-p_{k-i-1}$. We also have $m_i\le n$.
Therefore $m_i \le b_i$ for all $i\in\{1,\dots,k-1\}$.
So we may multiply $x_n$ on the right by canonical factors to obtain
a word $y_n$ of positive generators such that no product of consecutive
letters equals $\Delta$ and such that, for each $i\in\{1,\dots,k-1\}$,
there are exactly $b_i$ factors of length $i$.

So $(q_n,y_n):=(b_k, b_k + b_{k-1},\dots, b_k + \dots + b_1, y_n)$ is
in $\Omega_0$.

As $n$ tends to infinity, $b_k$ converges to $p_0$ and $b_i$ converges to
$p_{k-i}-p_{k-i-1}$ for $1\le i\le k-1$.
So $\sum_{i=0}^{j} b_{k-i}$ converges to $p_j$ for $j\in\{0,\dots,k-1\}$.
We also have that $y_n$ converges to $z$.
We conclude that $(q_n,y_n)$ converges to $(p,z)$, which must therefore
be in the closure of $\Omega_0$.
\end{proof}

\begin{lemma}
\label{lem:injective}
The map $\psi:\Omega\to \Z^{\dihed}$ is injective.
\end{lemma}
\begin{proof}
Let $(p,z)\in\Omega'$ and define $f(c):=\psi_{p,z}(\Delta^{-c})$
for all $c\in\Z$.
Since $\phi_i(\Delta^{-c},z)=c$ for all $0\le i\le k-1$, we have
\begin{align*}
f(c)= \sum_{i=0}^{k-1} |p_i + c| - \sum_{i=0}^{k-1} |p_i|.
\end{align*}
Observe that, for $x\in \N$,
\begin{align}
\label{eqn:diff}
|x+c+1|-|x+c| =
\begin{cases}
1, & \text{if $x\ge -c$}, \\
-1, & \text{otherwise}.
\end{cases}
\end{align}
So
\begin{align*}
f(c+1)-f(c) &= \sharp\{i \mid p_i\ge -c\} - \sharp\{i \mid p_i < -c\} \\
 &= 2 \sharp\{i \mid p_i\ge -c\} - k.
\end{align*}
Here $\sharp$ denotes the cardinal number of a set.
Therefore, by calculating $\psi_{p,z}(\Delta^{-c-1})-\psi_{p,z}(\Delta^{-c})$
for each $c\in\N$, we may determine the number of components of $p$
that equal each element of $\Z\union\{-\infty,+\infty\}$.
 Since the components of $p$ are non-decreasing,
we will then have determined $p$. Thus we have shown that if $(p_1,z_1)$
and $(p_2,z_2)$ are elements of $\Omega'$ such that $p_1\neq p_2$,
then $\psi_{p_1,z_1} \neq \psi_{p_2,z_2}$.

Now assume that $p_1=p_2=:p$
but that $(p,z_1)$ and $(p,z_2)$ are elements of distinct equivalence
classes in $\Omega$.
So, $p$ cannot be identically $+\infty$ or identically $-\infty$.
We know from Lemma~\ref{lem:bijection}
that $\psi$ is a bijection between $\Omega_0$
and $\dists$, so we may assume that not all entries of $p$ are finite
and that $z_1$ is an infinite word.
Let $x_n$ be the $n$th canonical factor of $z_1$ and let $w_n$ be the
product $w_n := x_1\cdots x_n$.

We deal first with the case where $p_0$ is finite.
For each canonical factor $y\in M^+$, denote by $l(y)$
the length of $y$, that is the total number of copies of $a$
and $b$ one has to multiply together to get $y$.
Observe that $\phi(w_n,z)-\phi(w_{n-1},z)= \phi(x_n,w_{n-1}^{-1}z)$
for any $z\in Z$. Since the effect of left multiplying $w_{n-1}^{-1}z_1$
by $x_n^{-1}$ is to cancel exactly one canonical factor of length
$l(x_n)$, we get
\begin{align}
\label{eqa}
\phi_i(w_n,z_1)-\phi_i(w_{n-1},z_1) =
   \begin{cases}
   0, & \text{if $i<k-l(x_n)$}, \\
   -1, & \text{otherwise},
   \end{cases}
\end{align}
for all $i\in\{0,\dots,k-1\}$ and $n\in\N$.
From~(\ref{eqn:diff}), we see that
\begin{align*}
\psi_{p,z_1}(w_n) - \psi_{p,z_1}(w_{n-1})
   &= \sum_{i=0}^{k-1} |p_i + \phi_i(w_n,z_1)|
             - \sum_{i=0}^{k-1} |p_i + \phi_i(w_{n-1},z_1)| \\
   &= - \sharp\{i \ge k-l(x_n) \mid p_i\ge -\phi_i(w_n,z_1) \} \\
   & \qquad\qquad   + \sharp\{i \ge k-l(x_n) \mid p_i <  -\phi_i(w_n,z_1) \}.
\end{align*}
Since we have assumed that $p_0$ is finite and not all components of $p$
are finite, we must have that $p_{k-1}=+\infty$. Therefore, the first set
above is not empty, and so
\begin{align}
\label{eqn:minustwo}
\psi_{p,z_1}(w_n) - \psi_{p,z_1}(w_{n-1}) \le l(x_n) -2,
\qquad\text{for all $n\in\N$.}
\end{align}

Now consider $z_2$. Since $z_2\neq z_1$, eventually some $x_n^{-1}$ will
not cancel completely with the first canonical factor of $w_{n-1}^{-1}z_2$
and subsequent left multiplications by $x_{n+1}^{-1},x_{n+2}^{-1},\dots$ will
have the effect of adding more factors.
For each $n\in\N$, let $y_n$ be such that $\Delta^{-1} y_n = x_n^{-1}$.
Since $y_n$ is a positive canonical factor of length $k-l(x_n)$, we get
\begin{align}
\label{eqn:eqb}
\phi_i(w_n,z_2)-\phi_i(w_{n-1},z_2) =
   \begin{cases}
   -1, & \text{if $i<l(x_n)$}, \\
   0, & \text{otherwise},
   \end{cases}
\end{align}
for all $i\in\{0,\dots,k-1\}$ and $n$ large enough.
So, for such $n$,
\begin{align*}
\psi_{p,z_2}(w_n) - \psi_{p,z_2}(w_{n-1})
   &= \sum_{i=0}^{k-1} |p_i + \phi_i(w_n,z_2)|
             - \sum_{i=0}^{k-1} |p_i + \phi_i(w_{n-1},z_2)| \\
   &= - \sharp\{i < l(x_n) \mid p_i\ge -\phi_i(w_n,z_2) \} \\
   &\qquad\qquad   + \sharp\{i < l(x_n) \mid p_i <  -\phi_i(w_n,z_2) \}.
\end{align*}
Let $i\in\{0,\dots,k-1\}$. If there are infinitely many $n\in\N$ such that
$i<l(x_n)$, then, by~(\ref{eqn:eqb}), the sequence $\phi_i(w_n,z_2)$ is
non-increasing and has limit $-\infty$.
But our assumption on $p$ implies that none
of the $p_i$ are equal to $-\infty$. Therefore, there are only a finite number
of $n\in\N$ such that the first set above contains $i$. Since this is true
for any $i$, the first set must eventually be empty.

So there are only finitely many $n$ for which
$\psi_{p,z_2}(w_n) - \psi_{p,z_2}(w_{n-1}) < l(x_n)$.
Comparing this with~(\ref{eqn:minustwo}), we see that $\psi_{p,z_1}$ and
$\psi_{p,z_2}$ cannot be equal.

Now suppose that $p_0=-\infty$.
Note that $\phi_i(w\Delta^{-c},z)=c+\phi_i(w,z)$ for all $w\in \dihed$ and
$0\le i \le k-1$.
So, using~(\ref{eqa}) and~(\ref{eqn:eqb}), we get
\begin{align*}
\phi_i(w_n\Delta^{-n},z_1)-\phi_i(w_{n-1}\Delta^{-n+1},z_1) =
   \begin{cases}
   1, & \text{if $i<k-l(x_n)$}, \\
   0, & \text{otherwise},
   \end{cases}
\end{align*}
for all $n\in\N$,
and
\begin{align*}
\phi_i(w_n\Delta^{-n},z_2)-\phi_i(w_{n-1}\Delta^{-n+1},z_2) =
   \begin{cases}
   0, & \text{if $i<l(x_n)$}, \\
   1, & \text{otherwise},
   \end{cases}
\end{align*}
for $n$  large enough.
Using similar logic to that of the preceding case, we can show that
\begin{align*}
\psi_{p,z_1}(w_n\Delta^{-n}) - \psi_{p,z_1}(w_{n-1}\Delta^{-n+1})
   \le l(x_n) -2
\end{align*}
for all $n\in\N$, and that
\begin{align*}
\psi_{p,z_2}(w_n\Delta^{-n}) - \psi_{p,z_2}(w_{n-1}\Delta^{-n+1})
   = l(x_n)
\end{align*}
for $n$ large enough. So in this case also, $\psi_{p,z_1}$ is different
from $\psi_{p,z_2}$.
\end{proof}

\begin{lemma}
\label{lem:continuous}
The map $\psi:\Omega\to\Z^{\dihed}$ is continuous.
\end{lemma}
\begin{proof}
Let $((p^{(n)},z^{(n)}))_{n\in\N}$ be a sequence in $\Omega$ converging to some element
$(p,z)$ of the same set in the topology we have chosen on $\Omega$.
If $(p,z)$ is in $\Omega_0$, then it is isolated and $(p^{(n)},z^{(n)})$ must
eventually be equal to it. So in this case, $\psi_{p^{(n)},z^{(n)}}$ obviously
converges to $\psi_{p,z}$.

Now suppose that $p=(\infty,\dots,\infty)$.
Observe that, for $w\in\dihed$ fixed, $\phi(w,z^{(n)})$ is bounded uniformly
in $n$. So, since each component of $p^{(n)}$
converges to $\infty$, we have, for each $w\in \dihed$, that
\begin{align*}
\psi_{p^{(n)},z^{(n)}}(w) = \sum_{i=0}^{k-1} \phi_i(w,z^{(n)}),
\qquad\text{for $n$ large enough}.
\end{align*}
But, by Lemma~\ref{lem:sumphi}, the right-hand-side is equal to
$\sum_{i=0}^{k-1}\pi_i(w^{-1})$, and this is exactly $\psi_{\plusclass}(w)$.

Similar reasoning shows that $\psi_{p^{(n)},z^{(n)}}$ converges to
$\psi_{\minusclass}$ if $p^{(n)}$ converges to $(-\infty,\dots,-\infty)$.

Suppose finally that $(p,z)$ is in $\Omega\backslash\Omega_0$ and $p$ is
not identically either $+\infty$ or $-\infty$. Then $z^{(n)}$ converges to
$z$ and so, by Lemma~\ref{lem:phiconverges}, $\phi(w,z^{(n)})$ converges to
$\phi(w,z)$ for each $w\in \dihed$. Since also $p^{(n)}$ converges to
$p$, we get that $\psi_{p^{(n)},z^{(n)}}$ converges to $\psi_{p,z}$
by inspecting the definition of $\psi$.
\end{proof}

\begin{theorem}
\label{thm:homeo}
The map $\psi$ is a homeomorphism between $\Omega$ and $\mathcal{M}$.
\end{theorem}
\begin{proof}
The injectivity of $\psi$ was proved in Lemma~\ref{lem:injective}
and so $\psi$ is a bijection from $\Omega$ to $\psi(\Omega)$.
As a continuous bijection from a compact space to a Hausdorff one,
$\psi$ must be a homeomorphism from $\Omega$ to $\psi(\Omega)$.
So $\psi(\Omega)$ is compact and therefore closed.
Since $\Omega=\closure\Omega_0$ by Lemma~\ref{lem:closure} and
$\psi$ is continuous by Lemma~\ref{lem:continuous},
we have $\psi(\Omega_0)\subset \psi(\Omega)\subset \closure\psi(\Omega_0)$.
Taking closures, we get $\psi(\Omega)=\closure\psi(\Omega_0)=\mathcal{M}$,
by Lemma~\ref{lem:bijection}.
\end{proof}

The proof of our characterisation of Busemann points will require
a result from~\cite{AGW-m}:
The Busemann points are precisely those horofunctions $\xi$
for which $H(\xi,\xi)=0$, where the \emph{detour cost} $H(\cdot,\cdot)$
is defined by
\begin{align*}
H(\xi,\eta):=
\liminf_{x\to \xi}
\big( d(b,x) + \eta(x) \big)
\end{align*}
for any pair of horofunctions $\xi$ and $\eta$.

\begin{theorem}
\label{thm:busemann}
A function in $\mathcal{M}$ is a Busemann point if and only if the
corresponding element $(p,z)$ of $\Omega$ is in $\Omega\backslash\Omega_0$
and satisfies the following:
$p_i-p_{i-1} = \emme_{k-i}(z)$ for every $i\in\{1,\dots,k-1\}$ such that
$p_i$ and $p_{i-1}$ are not both $-\infty$ nor both $+\infty$.
\end{theorem}
\begin{proof}
Assume $\xi\in\mathcal{M}$ is a Busemann point.
So $\xi$ is the limit of a sequence of group elements $x_n := y_0\cdots y_n$,
where $y$ is an infinite geodesic word.
Write $x_{n-1} = \Delta^r z_1\cdots z_s$ in left normal form
and let $j$ be the length of the last canonical factor $z_s$.
Consider the effect of right multiplying by $y_n$.
There are four cases, corresponding to the four elements of $S$:
\newcommand\condaa{i}
\newcommand\condab{ii}
\newcommand\condac{iii}
\newcommand\condad{iv}
\begin{itemize}
\renewcommand{\labelitemi}{\condaa.}
\item
$y_n$ is positive and $z_s y_n \in M^+\union\{\Delta\}$.
In this case the length of the last canonical factor increases by one and so
$\pi_{k-j-1}(x_n) = \pi_{k-j-1}(x_{n-1}) +1$.
All other components of $\pi(x_n)$ equal those of $\pi(x_{n-1})$;
\renewcommand{\labelitemi}{\condab.}
\item
$y_n$ is positive and $z_s y_n \not\in M^+\union\{\Delta\}$. 
In this case another canonical factor $y_n$ of length one is tacked onto
the end and so $\pi_{k-1}(x_n) = \pi_{k-1}(x_{n-1}) +1$, all other components
being the same;
\renewcommand{\labelitemi}{\condac.}
\item
$y_n$ is negative and $z_s y_n \in M^+\union\{e\}$.
In this case the length of the last canonical factor decreases by one and so
$\pi_{k-j}(x_n) = \pi_{k-j}(x_{n-1}) -1$, all other components
being the same;
\renewcommand{\labelitemi}{\condad.}
\item
$y_n$ is negative and $z_s y_n \not\in M^+\union\{e\}$.
In this case we can see what happens more clearly by right multiplying
$x_n$ by $\Delta^{-1}(\Delta y_n)$ instead of $y_n$. Moving the $\Delta^{-1}$
all the way to the left, we see that the power of $\Delta$ becomes $r-1$,
each canonical factor $z_i$;~$1\le i \le s$ is replaced by $\tau(z_i)$,
and another canonical factor $\Delta y_n$ of length $k-1$ is tacked onto 
the end. So $\pi_{0}(x_n) = \pi_{0}(x_{n-1}) -1$ and all other components
stay the same.
\end{itemize}

In all cases, when going from $\pi(x_{n-1})$ to $\pi(x_{n})$,
a single component is changed, either increased of decreased by one.
Looking at the distance formula of Proposition~\ref{prop:artindist},
we see that, since $y$ is a geodesic word, an increase is only possible
when the relevant component of $\pi(x_{n-1})$ is non-negative,
and a decrease is only possible when it is non-positive.

If case~({\condaa}) occurs infinitely often with $j=k-1$, then
$\pi_{0}(x_n)$ converges to $+\infty$ as $n$ tends to infinity, and
so every component of $\pi(x_n)$ converges to $+\infty$.
In this case, the condition in the statement of the theorem holds trivially.
So we may assume that case~({\condaa}) occurs only finitely many times
with $j=k-1$. Likewise, we may assume that case~({\condac}) occurs only
finitely many times with $j=1$.

One sees that case~({\condab}) creates a new canonical factor of length
one, which can be lengthened by successive applications of case~({\condaa}),
whereas case~({\condad}) creates a new canonical factor of length
$k-1$, which can be shortened by successive applications of case~({\condac}).
For each $n\in\N$, denote by $z^{(n)}$ the word consisting of all the canonical
factors of $x_n$ taken in sequence.
Because of the assumptions of the previous paragraph, eventually, once a
canonical factor has been created it can not be removed.
So if we take the sequence of times $(n_t)_{t\in\N}$ where either
case~({\condab}) or case~({\condad}) occurs, then the difference
between $z^{(n_t)}$ and $z^{(n_{t-1})}$ is that a new canonical factor has
been added and, possibly, that the original canonical factors have been
operated on by $\tau$.

Fix $i\in\{1,\dots,k-1\}$ such that $p_{i-1}$ and $p_i$ are not both $+\infty$
nor both $-\infty$. We have that $\pi_i(x_{n_t})-\pi_{i-1}(x_{n_t})$ is equal
to $m^{n_t}_{k-i}$, the number of canonical factors of length $k-i$ in
$z^{(n_t)}$. But because $z^{(n_t)}$ grows monotonically as $t$ increases,
$m^{n_t}_{k-i}$ converges as $t$ tends to infinity to $m_{k-i}(z)$,
the number of canonical factors of length $k-i$ in $z$.
Therefore,
\begin{align*}
p_i-p_{i-1} = \lim_{t\to\infty} (\pi_i(x_{n_t})-\pi_{i-1}(x_{n_t}))
            = \lim_{t\to\infty} m^{n_t}_{k-i}
            = m_{k-i}(z).
\end{align*}
This establishes the implication in one direction.

Now assume that $\xi\in\mathcal{M}$ corresponds to
$\plusclass$.
For each $n\in\N$, let $x_n:=\prodd(a,b;n)$
and let $(p^{(n)},z^{(n)})$ be the corresponding element of $\Omega_0$.
We see that $p^{(n)}_0 = \lfloor n/k \rfloor$, which
tends to infinity as $n$ tends to infinity. It follows that $p^{(n)}$ converges
to $(+\infty,\dots,+\infty)$ and hence $x_n$ converges to $\xi$
by Theorem~\ref{thm:homeo}.
Since $x_n$ is a geodesic, $\xi$ must be a Busemann point.

When $\xi\in\mathcal{M}$ corresponds to $\minusclass$,
we take $x_n:=\prodd(a^{-1},b^{-1};n)$ and use a similar argument.

Now assume that $\xi$ corresponds to an element
$(p,z)\in\Omega\backslash(\Omega_0\union\{\minusclass,\plusclass\})$
satisfying the condition in the statement of the theorem.
Let $w^n$ be the word consisting of the first $n$ canonical factors of $z$.
Let $j\in\{0,\dots,k-1\}$ be the index of either the first non-negative
component of $p$ or the last non-positive component.
We can choose a sequence of vectors $q^n$ in $\Z^k$
such that $q^n_i-q^n_{i-1}=m_{k-i}(w^n)$ for all $i\in\{1,\dots,k-1\}$,
and such that $q^n_j$ converges to $p_j$.
Since $(q^n,w^n)\in\Omega_0$ for all $n\in\N$, we may consider the
element $x^n$ of $\dihed$ corresponding to $(q^n,w^n)$.
From our assumption on $\xi$, we have that $m_{k-i}(w^n)$ converges
as $n$ tends to infinity to $p_i-p_{i-1}$ for all $i\in\{1,\dots,k-1\}$
such that $p_i$ and $p_{i-1}$ are not both $+\infty$ nor both $-\infty$.

Using this and the definition of $q^n$, we conclude that $q^n$ converges
to $p$ as $n$ tends to infinity. But we also have that $w^n$ converges to $z$
and so, by Theorem~\ref{thm:homeo}, $x^n$ converges to $\xi$.
Multiplying $z$ on the left by $(x^n)^{-1}$ has the effect of canceling
$m_i(w^n)$ factors of length $i$ for each $i\in\{1,\dots,k-1\}$ and
adding a factor $\Delta^{-q^n_0}$. Therefore
\begin{align*}
\phi_i(x^n,z) = -q^n_0 - m_{k-1}(w^n) - \dots - m_{k-i}(w^n) = -q^n_i.
\end{align*}
So
\begin{align*}
H(\xi,\xi) &\le \liminf_{n\to\infty}(\dist(e,x^n)+\psi_{p,z}(x^n)) \\
           &= \liminf_{n\to\infty}\sum_{i=0}^{k-1}
                          \Big(|q_i^n| + |p_i - q_i^n| - |p_i| \Big) \\
           &= 0,
\end{align*}
since $q^n$ converges to $p$.
This proves that $\xi$ is a Busemann point.
\end{proof}

\section{Dual generators}

We establish a formula for the dual-generator word-metric using a technique
originally developed by Fordham~\cite{fordham} to prove a length formula
for Thompson's group $F$.
The following theorem is a right-handed version of one in~\cite{belk_brown}.
\begin{theorem}
\label{thm:lengthtechnique}
Let $G$ be a group with generating set $S$, and let $l:G\to\N$ be a
function. Then $l$ gives the distance with respect to $S$ from the identity
to any given element if and only if
\begin{itemize}
\renewcommand{\labelitemi}{L1.}
\item
$l(e)=0$,
\renewcommand{\labelitemi}{L2.}
\item
$|l(wg)-l(w)|\le 1$ for all $w\in G$ and $g\in S$,
\renewcommand{\labelitemi}{L3.}
\item
if $w\in G\backslash\{e\}$, then there exists $g\in S\union S^{-1}$
such that $l(wg)<l(w)$.
\end{itemize}
\end{theorem}

\begin{proposition}
\label{prop:dualdistance}
Let $w=\delta^{r} w_1 w_2\cdots w_s$ be written in left normal
form with respect to the dual generators.
Then the distance between the identity and $w$ with respect to
these generators is given by $\dualdist(e,w)=|r|+|r+s|$.
\end{proposition}
\begin{proof}
Let $l(w):= |r|+|r+s|$. Clearly $l$ satisfies~(L1). Consider
the effect of right multiplying $w$ by a generator $g\in \dualgens$.
Let $v:=wg$ and write this group element in left normal form
$v=\delta^{r'} v_1 v_2\cdots v_{s'}$.
There are four cases to consider:
\begin{itemize}
\renewcommand{\labelitemi}{\conda.}
\item
$g$ is positive and $w_sg=\delta$. In this case $r'=r+1$ and $s'=s-1$.
\renewcommand{\labelitemi}{\condb.}
\item
$g$ is positive and $w_sg\neq \delta$. In this case $r'=r$ and $s'=s+1$.
\renewcommand{\labelitemi}{\condc.}
\item
$g$ is negative and $w_sg=e$. In this case $r'=r$ and $s'=s-1$.
\renewcommand{\labelitemi}{\condd.}
\item
$g$ is negative and $w_sg\neq e$. In this case $r'=r-1$ and $s'=s+1$.
\end{itemize}
In all cases, either $r'=r$ and $r'+s'=r+s \pm 1$,
or $r'=r\pm 1$ and $r'+s'=r+s$. Therefore~(L2) is satisfied.

Also, by choosing $g$ appropriately, we can make whichever of the four
cases we want happen. So we always have the freedom to increase
or decrease either $r$ or $r+s$ by one. It follows that~(L3) holds.
\end{proof}

We note that an algorithm for finding a geodesic representative
of any given word in $A_3$ with respect to the dual generators was
presented in~\cite{xu_genus}.

Observe that the distance formula above has a form similar to
the formula established in Proposition~\ref{prop:artindist}
for the distance with respect to the Artin generators.
This similarity will allow us to calculate the horofunction boundary
and the Busemann points with respect to the dual generators using
the same method as for the Artin generators.

\newcommand\dualcownt{\tilde m}
\newcommand\dualpi{\tilde \pi}
\newcommand\dualphi{\tilde \phi}
\newcommand\dualpsi{\tilde \psi}
\newcommand\dualpara{\tilde \Omega}

As before we define some maps.
For any $w\in\dihed$, let $\dualcownt_1(w)$ and $\dualcownt_2(w)$ be such
that $w$ can be written in left normal form as
$w=\delta^{\dualcownt_2(w)} w_1\cdots w_{\dualcownt_1(w)}$.
Define $\dualpi:\dihed\to \Z^2$ by
\begin{align*}
\dualpi(w) := (\dualcownt_2(w), \dualcownt_1(w) + \dualcownt_2(w) ).
\end{align*}
Finally, let
\begin{align*}
\dualphi(w,z):= \dualpi(w^{-1}z) - \dualpi(z),
\qquad\text{for all $w$ and $z$ in $\dihed$.}
\end{align*}

The proof of the following lemma is similar to its counterpart,
Lemma~\ref{lem:phiconverges}.
\begin{lemma}
Let $w\in\dihed$ and let $z_1z_2\cdots$ be an infinite word of positive
dual generators such that no product of consecutive letters equals
$\delta$.
Then $\dualphi(w,z_1\cdots z_n)$ converges as $n$ tends to infinity.
\end{lemma}

Let $\dualZ$ be the set of possibly infinite words of positive dual generators
having no product of consecutive letters equal to $\delta$.
The previous lemma allows us to define $\dualphi(w,z)$ for $w\in\dihed$ and
$z=z_1 z_2\cdots$ an infinite element of $\dualZ$ to be the limit of
$\dualphi(w,z_1\cdots z_n)$ as $n$ tends to infinity.

Let $\dualpara'$ denote the set of $(p,z)$ in
$(\Z\union\{-\infty,+\infty\})^2 \times \dualZ$ such that if
$p$ is not identically $-\infty$ nor identically $+\infty$,
then $p_1-p_0=\dualcownt_1(z)$.
We take the product topology on $\dualpara'$.
Let $\dualpara$ be the quotient of $\dualpara'$
obtained by considering all points $((-\infty,-\infty),z)$ with $z\in \dualZ$
to be equivalent,
and all points $((+\infty,+\infty),z)$ with $z\in \dualZ$ to be equivalent.
The former equivalence class we denote simply by $-\tilde\infty$, the
latter by $+\tilde\infty$.

For each $(p,z)\in\dualpara'$, define
\begin{align}
\label{dualpsidefinition}
\dualpsi_{p,z} : \dihed \to \Z, \quad
       w\mapsto |p_0+\dualphi_0(w,z)| + |p_1+\dualphi_1(w,z)| - |p_0| - |p_1|.
\end{align}
We use the same convention as before for adding and subtracting infinities.
The following lemma shows that $\dualpsi$ is constant on the equivalence
classes $-\tilde\infty$ and $+\tilde\infty$.
The proof of this lemma is the same as that of Lemma~\ref{lem:sumphi}.
\begin{lemma}
For all $w$ and $z$ in $\dihed$,
\begin{align*}
\dualphi_0(w,z) + \dualphi_1(w,z)= \dualpi_0(w^{-1}) + \dualpi_1(w^{-1}).
\end{align*}
\end{lemma}
So we see that if $p=(-\infty,-\infty)$, then
\begin{align*}
\dualpsi_{p,z}(w)=-\dualpi_0(w^{-1}) - \dualpi_1(w^{-1})
\end{align*}
is independent of $z$. Likewise, if $p=(+\infty,+\infty)$, then
\begin{align*}
\dualpsi_{p,z}(w)=\dualpi_0(w^{-1}) + \dualpi_1(w^{-1}).
\end{align*}
We may therefore consider the map $\dualpsi$ to be defined on $\dualpara$.

Let $\dualdists:=\{\dualdist(\cdot,x)-\dualdist(e,x) \mid x\in \dihed \}$
and let $\dualcompactification$ be its closure, that is,
the horofunction compactification
of $\dihed$ with the dual-generator word metric.

Let $\dualfinites$ be the set of finite words with letters in
$\{\sigma_1,\dots,\sigma_k\}$ having no product of consecutive
letters equal to $\delta$ and define
\begin{align*}
\dualpara_0:= \{ (p,z) \in \Z^2\times \dualfinites \mid
                   p_1-p_0 = \dualcownt_1(z) \}.
\end{align*}

Again, we wish to show that $\dualpara$ is homeomorphic to
$\dualcompactification$ with $\dualpara_0$ being mapped to $\dualdists$.
We use the same method we used for the Artin generators.
The proofs of the following results are similar to those of the corresponding
results in Section~\ref{sec:artin}.
\begin{lemma}
Restricted to $\dualpara_0$, the map $\dualpsi$ is a bijection between
$\dualpara_0$ and $\dualdists$.
\end{lemma}
\begin{lemma}
The set $\dualpara_0$ is dense in $\dualpara$.
\end{lemma}
\begin{lemma}
The map $\dualpsi:\dualpara\to \Z^{\dihed}$ is injective.
\end{lemma}
\begin{lemma}
The map $\dualpsi:\dualpara\to\Z^{\dihed}$ is continuous.
\end{lemma}
\begin{theorem}
The map $\dualpsi$ is a homeomorphism between $\dualpara$ and
$\dualcompactification$.
\end{theorem}

The proof of the following theorem uses the same reasoning as that of
Theorem~\ref{thm:busemann}.
\begin{theorem}
In the horoboundary of $\dihed$ with the dual-generator word metric,
all horofunctions are Busemann points.
\end{theorem}

We use our distance formula to characterise the geodesic words of $\dihed$
with the dual generators.
\begin{proposition}
\label{dualgeodesics}
Let $x\in \dihed$ and let $y$ be a freely reduced word of dual generators
representing $x$. Then $y$ is a geodesic if and only if
$\dualposs(y)+\dualnegg(y) \le 2$.
\end{proposition}
\begin{proof}
Let $y$ be such that $\dualposs(y)+\dualnegg(y)>2$.
Since neither $\dualposs(y)$ nor $\dualnegg(y)$ are greater than $2$,
one of them must equal $2$ and the other must be positive.
Suppose $\dualposs(y)=2$ and $\dualnegg(y)>0$. Then $y$ contains a negative
generator and two consecutive positive generators with product $\delta$.
Take the $\delta$ and shift it towards the negative generator by repeatedly
using the relations $\sigma_i\delta=\delta\sigma_{i+2}$
and $\sigma^{-1}_i\delta=\delta\sigma^{-1}_{i+2}$.
Then cancel the negative generator with the $\delta$ using
$\sigma^{-1}_i\delta=\sigma_{i+1}$.
The result is a word representing $x$ that is shorter by one generator
than $y$. Therefore $y$ is not a geodesic. The proof in the case when
$\dualposs(y)>0$ and $\dualnegg(y)=2$ is similar.

Now assume that $\dualnegg(y)=0$. Consider what happens if we start
at the identity and successively multiply by generators as prescribed by
$y$. We obtain a sequence, which we denote by $(x_n)_{n\in\N}$.
Initially $r=r+s=0$, where $r$ and $s$ are as in
Proposition~\ref{prop:dualdistance}.
Since $y$ is composed only of positive generators, only cases~(\conda)
and~(\condb) in the proof of Proposition~\ref{prop:dualdistance}
are relevant here. We note that in these two cases, either $r$ or $r+s$
increases by one, and the other stays the same.
Therefore $\dualdist(e,x_n)=n$. It follows that $y$ is a geodesic.

The proof that $y$ is a geodesic if $\dualposs(y)=0$ is similar.
The cases concerned this time are~(\condc) and~(\condd),
and in both of these either $r$ or $r+s$ decreases by one and the
other stays the same.

The final case to consider is when $\dualposs(y)=\dualnegg(y)=1$.
We claim that, as the generators comprising $y$ are successively multiplied,
the rightmost canonical factor in the left normal form of $x_n$
is equal to $y_n$ when $y_n$ is positive and equal to $\delta y_n$ when
$y_n$ is negative. To show this, we use induction on $n$.
Suppose the claim is true for $x_n$, which
we write in left normal form as $x_n=\delta^r w_1\cdots w_s$.
If $y_n$ is positive, our induction hypothesis gives that $w_s=y_n$,
and so $y_{n+1}$ can not equal either $w_s^{-1}$ or $w_s^{-1}\delta$
since $y$ is freely reduced and $\dualposs(y)<2$.
Therefore, if $y_n$ is positive, neither case~(\conda) nor case~(\condc)
of Proposition~\ref{prop:dualdistance} can occur. Since there is no
cancellation, the left normal form of $x_{n+1}$ has then
$y_{n+1}$ or $\delta y_{n+1}$ as rightmost canonical factor,
depending on whether $y_{n+1}$ is positive or negative.
Similar reasoning shows the same is true when $y_n$ is negative.
Thus we have proved our claim.

The argument of the previous paragraph also established that cases~(\conda)
and~(\condc) of Proposition~\ref{prop:dualdistance} never occur when
$x_n$ is multiplied on the right by $y_{n+1}$.

In case~(\condb) of that proposition, $r+s$ increases by one while $r$
remains the same, and in case~(\condd), $r$ decreases by one while $r+s$
remains the same. Therefore, $|r|+|r+s|$ always increases by one as each letter
of $y$ is added, and so $\dualdist(e,x_n)=n$.
So in this case also, $y$ is a geodesic.
\end{proof}

This characterisation of geodesics allows us to calculate the
geodesic growth series of $\dihed$.
\begin{theorem}
The geodesic growth series of $\dihed$ with the dual generators is
\begin{align*}
\grote(x) =
      \frac{1+(3-2k)x + (2 + k^2 -3k) x^2 - 2k(k-1)x^3}
           {(1-kx)(1-2(k-1)x)(1-(k-1)x)}.
\end{align*}
\end{theorem}
\begin{proof}
Let $N^n_{ij}$ be the number of freely reduced words $y$ of length $n$
satisfying $\dualposs(y)\le i$ and $\dualnegg(y) \le j$,
and let $\grote_{ij}$ be the corresponding generating series.
Proposition~\ref{dualgeodesics} and an inclusion--exclusion argument
give that the number of geodesics of length $n$ is
\begin{align*}
N^n_{20} + N^n_{02} + N^n_{11} - N^n_{10} - N^n_{01}.
\end{align*}
Therefore
\begin{align}
\label{eqn:incexc}
\grote = \grote_{20} + \grote_{02} + \grote_{11} - \grote_{10} - \grote_{01}.
\end{align}
Clearly, $N^n_{20} = N^n_{02} = k^n$ for all $n\in \N$, and so
\begin{align*}
\grote_{20}(x) = \grote_{02}(x) = 1 + kx + k^2 x^2 + \dots
            = \frac{1}{1-kx}.
\end{align*}

Consider now the freely reduced words not containing $\delta$ or $\delta^{-1}$
as sub-words.
For the first letter we may choose any of the $2k$ generators.
For subsequent letters, we can choose any letter apart from the inverse
of the previous one and the letter that would combine with the previous one
to form $\delta$ or $\delta^{-1}$. So we have a choice of $2k-2$
generators. Therefore the growth series $\grote_{11}$ for this set of words is
\begin{align*}
\grote_{11}(x) &= 1 + 2kx + 2k(2k-2) x^2 + 2k(2k-2)^2 x^3 + \cdots \\
            &= \frac{1+2x}{1-2(k-1)x}.
\end{align*}

Now consider the set of freely reduced words containing only positive
generators and
no sub-word equal to $\delta$. This time there are $k$ possibilities for the
first letter and $k-1$ for subsequent letters. So the growth series is
\begin{align*}
\grote_{10}(x) &= 1 + kx + k(k-1) x^2 + k(k-1)^2 x^3 + \cdots \\
            &= \frac{1+x}{1-(k-1)x}.
\end{align*}
The growth series $\grote_{01}$ is identical.

The conclusion now follows from~(\ref{eqn:incexc}) after some rearranging.
\end{proof}
The first few terms of $\grote(x)$ are
\begin{align*}
\grote_(x) = 1 + 2kx + 2(2k^2-k)x^2 + 2(k^3+3k(k-1)^2)x^3
             + \cdots.
\end{align*}

\bibliographystyle{plain}
\bibliography{artin}

\begin{thebibliography}{10}

\bibitem{AGW-m}
Marianne Akian, St\'ephane Gaubert, and Cormac Walsh.
\newblock {The max-plus Martin boundary}.
\newblock Preprint. arXiv:math.MG/0412408, 2004.

\bibitem{belk_brown}
James~M. Belk and Kenneth~S. Brown.
\newblock Forest diagrams for elements of {T}hompson's group {$F$}.
\newblock {\em Internat. J. Algebra Comput.}, 15(5-6):815--850, 2005.

\bibitem{berger_braids}
Mitchell~A. Berger.
\newblock Minimum crossing numbers for {$3$}-braids.
\newblock {\em J. Phys. A}, 27(18):6205--6213, 1994.

\bibitem{bessis_dual}
David Bessis.
\newblock The dual braid monoid.
\newblock {\em Ann. Sci. \'Ecole Norm. Sup. (4)}, 36(5):647--683, 2003.

\bibitem{charney_meier_language}
Ruth Charney and John Meier.
\newblock The language of geodesics for {G}arside groups.
\newblock {\em Math. Z.}, 248(3):495--509, 2004.

\bibitem{connes}
A.~Connes.
\newblock Compact metric spaces, {F}redholm modules, and hyperfiniteness.
\newblock {\em Ergodic Theory Dynam. Systems}, 9(2):207--220, 1989.

\bibitem{coornaert_papadopoulos_horofunctions}
Michel Coornaert and Athanase Papadopoulos.
\newblock Horofunctions and symbolic dynamics on {G}romov hyperbolic groups.
\newblock {\em Glasg. Math. J.}, 43(3):425--456, 2001.

\bibitem{develin_cayley}
Mike Develin.
\newblock Cayley compactifications of abelian groups.
\newblock {\em Ann. Comb.}, 6(3-4):295--312, 2002.

\bibitem{word_processing}
David B.~A. Epstein, James~W. Cannon, Derek~F. Holt, Silvio V.~F. Levy,
  Michael~S. Paterson, and William~P. Thurston.
\newblock {\em Word processing in groups}.
\newblock Jones and Bartlett Publishers, Boston, MA, 1992.

\bibitem{fordham}
S.~Blake Fordham.
\newblock Minimal length elements of {T}hompson's group {$F$}.
\newblock {\em Geom. Dedicata}, 99:179--220, 2003.

\bibitem{gromov:hyperbolicmanifolds}
M.~Gromov.
\newblock Hyperbolic manifolds, groups and actions.
\newblock In {\em Riemann surfaces and related topics: Proceedings of the 1978
  Stony Brook Conference (State Univ. New York, Stony Brook, N.Y., 1978)},
  volume~97 of {\em Ann. of Math. Stud.}, pages 183--213, Princeton, N.J.,
  1981. Princeton Univ. Press.

\bibitem{mairesse_matheus_growth}
Jean Mairesse and Fr\'ed\'eric Math\'eus.
\newblock Growth series for {A}rtin groups of dihedral type.
\newblock Preprint, 2005.

\bibitem{rieffel_group}
Marc~A. Rieffel.
\newblock Group {$C\sp *$}-algebras as compact quantum metric spaces.
\newblock {\em Doc. Math.}, 7:605--651 (electronic), 2002.

\bibitem{sabalka_geodesics}
Lucas Sabalka.
\newblock Geodesics in the braid group on three strands.
\newblock In {\em Group theory, statistics, and cryptography}, volume 360 of
  {\em Contemp. Math.}, pages 133--150. Amer. Math. Soc., Providence, RI, 2004.

\bibitem{storm_barycenter}
Peter~A. Storm.
\newblock {The barycenter method on singular spaces}.
\newblock Preprint. arXiv:math.GT/0301087, 2003.

\bibitem{winweb_hyperbolic}
Corran Webster and Adam Winchester.
\newblock Boundaries of hyperbolic metric spaces.
\newblock {\em Pacific J. Math.}, 221(1):147--158, 2005.

\bibitem{winweb_busemann}
Corran Webster and Adam Winchester.
\newblock Busemann points of infinite graphs.
\newblock {\em Trans. Amer. Math. Soc.}, 358(9):4209--4224 (electronic), 2006.

\bibitem{xu_genus}
Peijun Xu.
\newblock The genus of closed {$3$}-braids.
\newblock {\em J. Knot Theory Ramifications}, 1(3):303--326, 1992.

\end{thebibliography}

\end{document}